\documentclass[a4paper]{amsart}

\usepackage{amsthm}
\usepackage{amsmath}
\usepackage{amssymb}
\usepackage{enumerate}

\theoremstyle{plain} \newtheorem{thm}{Theorem}[section]
\theoremstyle{plain} \newtheorem{prop}[thm]{Proposition}
\theoremstyle{plain} \newtheorem{lemma}[thm]{Lemma}
\theoremstyle{plain} \newtheorem{cor}[thm]{Corollary}
\theoremstyle{plain} 
\theoremstyle{plain} \newtheorem{question}[thm]{Question}
\theoremstyle{definition} \newtheorem{defin}[thm]{Definition}
\theoremstyle{definition} \newtheorem{notation}[thm]{Notation}
\theoremstyle{remark} \newtheorem{rem}[thm]{Remark}

\newcommand{\mathdef}{\stackrel{\textrm{\scriptsize def}}{=}}
\newcommand{\hdim}{\mathrm{dim}}
\newcommand{\mdim}{{\overline{\mathrm{dim}}_{\mathrm{M}}}}
\newcommand{\R}{\mathbb{R}}
\newcommand{\Q}{\mathbb{Q}}
\newcommand{\Rn}{\mathbb{R}^n}
\newcommand{\N}{\mathbb{N}}
\newcommand{\Hd}{\mathcal{H}}
\newcommand{\Hh}{\widehat{ \mathcal{H} }}
\newcommand{\cont}{\mathfrak{c}}
\newcommand{\al}{\alpha}
\newcommand{\be}{\beta}
\newcommand{\ga}{\gamma}
\newcommand{\de}{\delta}
\newcommand{\eps}{\varepsilon}
\newcommand{\su}{\subset}
\newcommand{\diam}{\mbox{diam}}
\newcommand{\Hc}{{\mathcal{H}}_\infty} 
\newcommand{\Ch}{\widetilde{C}} 

\title{How large dimension guarantees a given angle?}

\author[Harangi]{Viktor Harangi}
\address{A.~R\'enyi Institute of Mathematics, 
Hungarian Academy of Sciences, 
P.O.B.~127, H-1364 Budapest, Hungary}
\email{bizkit@cs.elte.hu}

\author[Keleti]{Tam\'as Keleti}
\address{Institute of Mathematics,
E\"otv\"os Lor\'and University,
P\'azm\'any P\'eter s.~1/c, 1117
Budapest, Hungary}
\email{tamas.keleti@gmail.com}
\urladdr{http://www.cs.elte.hu/analysis/keleti}

\author[Kiss]{Gergely Kiss}
\address{Institute of Mathematics,
E\"otv\"os Lor\'and University,
P\'azm\'any P\'eter s.~1/c, 1117
Budapest, Hungary}
\email{kisss@cs.elte.hu}

\author[Maga]{P\'eter Maga}
\address{Department of Mathematics and its Applications, Central European University, 
N\'ador u.~9., H-1051 Budapest, Hungary}
\email{magapeter@gmail.com}

\author[M\'ath\'e]{Andr\'as M\'ath\'e}
\address{Mathematics Institute, University of Warwick,
Coventry, CV4~7AL, United Kingdom}
\email{A.Mathe@warwick.ac.uk}

\author[Mattila]{Pertti Mattila}
\address{Department of Mathematics and Statistics, 
University of Helsinki, 
P.O.B.~68, FI-00014 Helsinki, Finland}
\email{pemattil@cc.helsinki.fi}

\author[Strenner]{Bal\'azs Strenner}
\address{Institute of Mathematics,
E\"otv\"os Lor\'and University,
P\'azm\'any P\'eter s.~1/c, 1117
Budapest, Hungary}
\email{strenner@cs.elte.hu}



\thanks{
The first, second and fifth authors were supported by
Hungarian Scientific Foundation grant no.~72655.
The second author was also supported by Bolyai J\'anos Scholarship.
The fifth author was also supported by EPSRC grant EP/G050678/1.
The sixth author would like to acknowledge the support of Academy of Finland.
}

\begin{document}

\begin{abstract}
We study the following two problems:

(1) Given $n\ge 2$ and $\al$, how large Hausdorff dimension can a compact
set $A\su\Rn$ have if $A$ does not contain three points that form an angle $\al$?

(2) Given $\al$ and $\de$, how large Hausdorff dimension can a 
subset $A$ of a Euclidean space have
if $A$ does not contain three points that form an angle
in the $\de$-neighborhood of $\al$?

An interesting phenomenon is that different angles show different behaviour in
the above problems. Apart from the clearly special extreme angles
$0$ and $180^\circ$, the angles $60^\circ,90^\circ$ and $120^\circ$ also
play special role in problem (2): the maximal dimension is smaller for these
special angles than for the other angles. 
In problem (1) the angle
$90^\circ$ seems to behave differently from other angles.
\end{abstract}

\maketitle

\section{Introduction}

The task of guaranteeing given patterns in a sufficiently large set has been a
central problem in different areas of mathematics for a long time.
Perhaps the most
famous example is the celebrated theorem of Szemer\'edi \cite{szemeredi},
which states that any sequence of positive integers with positive upper density
contains arbitrarily long arithmetic progressions.

More closely related to the present paper are the results of
Falconer \cite{falconer}, Keleti \cite{keleti} and Maga \cite{maga}, 
which state that for any three points in $\R$ or in $\R^2$ 
there exists a set of full Hausdorff dimension 
that contains no similar copy to the three given points.
It is open whether the analogous result holds in higher dimension.
In case of a negative answer it would be natural to ask what Hausdorff
dimension guarantees a similar copy of three given points.
Since the similar copy of a triangle has the same angles as the original
one, the following question arose.

\begin{question}\label{question}
For given $n$ and $\al$, what is the smallest $d$ for which
any compact set $A\su \Rn$ with Hausdorff dimension larger than $d$
contains three points that form an angle $\al$?
\end{question}

We use the following terminology.

\begin{defin}
We say that the set $A \subset \R^n$ \emph{contains the angle} $\alpha \in [0,180^\circ]$
if there exist distinct points $x, y, z \in A$ such that the angle between the vectors $y-x$ and $z-x$ is $\alpha$.
\end{defin}

\begin{defin}\label{def:basicdef}
If $n\ge 2$ is an integer and $\alpha \in [0,180^\circ]$, then let
\begin{multline*}
C(n, \alpha)= \sup \{s : \exists A \subset \R^n
\textrm{ compact such that } \\ \hdim(A)=s
\textrm{ and } A \textrm{ does not contain the angle } \alpha \}.
\end{multline*}
\end{defin}

Clearly, answering Question~\ref{question} is the same as finding $C(n, \al)$.
Somewhat surprisingly our results highly depend on the given angle.
For $90^\circ$ we show (Theorem~\ref{n/2}) that $C(n,90^\circ)\le [(n+1)/2]$
(where $[a]$ denotes the integer part of $a$) while for other angles
we prove (Theorem~\ref{n-1}) only $C(n,\alpha)\le n-1$, which is sharp
for $\al=0$ and $\al=180^\circ$.

In the other direction, the fifth author (M\'ath\'e) constructed 
compact sets of Hausdorff dimension $n/8$ not containing $\alpha$; 
this construction is published separately in \cite{mathe}. 
He obtains a better result ($n/4$) in the special case 
when $\cos^2 \alpha$ is rational, 
and an even better one ($n/2$) when $\alpha = 90^\circ$. 
Table \ref{table1} shows the best known bounds for $C(n,\alpha)$.  

\begin{table}[h]
\caption{Best known bounds for $C(n, \alpha)$}
\label{table1}
\centering
\begin{tabular}{l | l | l}
$\alpha$ & lower bound & upper bound\\
[0.5ex]       
\hline
$0, 180^\circ$ & $n-1$ & $n-1$ \\
$90^\circ$ & $n/2$ \cite[Thm 3.1]{mathe} & $[(n+1)/2]$\\
$\cos^2 \alpha \in \Q$ & $n/4$ \cite[Thm 3.2]{mathe} & $n-1$\\
other angles & $n/8$ \cite[Thm 3.4]{mathe} & $n-1$ \\
[1ex]
\end{tabular}
\end{table}

In the present paper for any 
$\alpha\in (0,180^\circ)\setminus \{60^\circ, 90^\circ, 120^\circ\}$ 
we construct (Theorem~\ref{thm:selfsimilar})
a self-similar compact set with Hausdorff dimension
$c(\al)\log n$ that does not contain the angle $\al$.
Although this is a much weaker result in terms of the dimension of the set, 
it has an advantage over M\'ath\'e's construction, namely, 
the constructed sets avoid not only $\alpha$, 
but also a small neighborhood of $\alpha$.  

In light of the above construction it is natural to ask 
what can be said if we only want to guarantee an angle
near to a given angle. In Section~\ref{approx} we show that the previously
mentioned special angles $(0, 60^\circ, 90^\circ, 120^\circ, 180^\circ)$
are really very special. 
If we fix $\al$ and a sufficiently small $\de$ (but do not fix $n$)
then for all other angles 
the above-mentioned self-similar construction (Theorem~\ref{thm:selfsimilar})
gives a compact set with arbitrarily large Hausdorff dimension 
that does not contain any angle from the $\de$-neighborhood of $\al$, 
while for the special angles this is not the case. More precisely,
we show that any 
set with Hausdorff 
dimension larger than $1$ contains angles arbitrarily close to the right angle 
(Theorem~\ref{Hausdorff1}), 
and that any 
set with Hausdorff dimension larger than
$\frac{C}{\delta}\log(\frac{1}{\delta})$ (with an absolute constant $C$)
contains angles from the $\delta$-neighborhoods of $60^\circ$ and
$120^\circ$ (Corollary~\ref{cor:60} and Theorem~\ref{thm:120}).
For the angles $0$ and $180^\circ$ it was already known by Erd{\H o}s and F\"uredi
\cite{erdos/furedi} that any infinite set contains angles
arbitrarily close to $0$ and angles arbitrarily close to $180^\circ$.

Note that in the previous results the dimension of
the Euclidean space ($n$) did not play any role.
To sum up the results we introduce the following function $\Ch$.
\begin{defin}\label{def:ch}
If $\alpha \in [0,180^\circ]$ and $\de>0$, then let
\begin{multline*} 
\Ch(\alpha, \delta) = \sup \{\hdim(A) :
A \subset \R^n \textrm{ for }\emph{some } n; \\
A \mbox{ does not contain any angle from } (\alpha-\delta, \alpha+\delta) \} .
\end{multline*}
\end{defin}

Theorem~\ref{thm:selfsimilar} implies that $\Ch(\alpha, \delta) = \infty$ 
if $\alpha$ is different from the special angles 
$0$, $60^\circ$, $90^\circ$, $120^\circ$, $180^\circ$ and 
$\delta$ is smaller than the distance of $\alpha$ from the special angles. 
A construction of the first author (Harangi \cite{harangi}) shows that 
$\Ch(\alpha, \delta) \geq \frac{c}{\delta}/\log(\frac{1}{\delta})$ 
for the angles $\alpha = 60^\circ, 120^\circ$.  
We summarize the above results in Table~\ref{table2}.
\begin{table}[h]
\caption{Smallest dimensions that guarantee an angle in the $\de$-neighborhood
of $\al$}
\label{table2}
\centering
\begin{tabular}{l | l | l}
$\alpha$ & $\Ch(\alpha, \delta)$ & \\
[0.5ex]       
\hline
$0, 180^\circ$ & $= 0$ & \\
$90^\circ$ & $= 1$ & \\
$60^\circ, 120^\circ$ & $\approx 1/\delta$ & apart from a multiplicative error $C\cdot\log(1/\delta)$ \\
other angles & $= \infty$ & provided that $\delta$ is sufficiently small\\
[1ex]
\end{tabular}
\end{table}

We emphasize the difference between the tasks of
finding an angle precisely and finding it approximately.
For example, we can find angles arbitrarily close to $90^\circ$
given that the dimension of the set is greater than $1$,
while if we want to find $90^\circ$ precisely in the set,
we need to know that its dimension is greater than $n/2$.

The following question is also closely related:
How large does the Hausdorff dimension of
a compact subset of $\R^n$ need to be to ensure that the set of angles 
contained in the set has positive Lebesgue measure? 
In \cite{iosevich} it is proved that larger than $\frac{n+1}2$ is enough and in 
\cite{mathe} that $n/6$ is not enough.

\begin{notation}
\label{not:Hausdorff}
We denote the $s$-dimensional Hausdorff measure by $\Hd^s$.

By \emph{$\dim$} we denote the Hausdorff dimension.

Recall that compact sets having the property $0<\Hd^s(K)<\infty$ are
called \emph{compact $s$-sets}.
\end{notation}

Using the well-known fact that an
analytic set $A$ with positive $\Hd^s$ measure contains a compact $s$-set
(see e.g. \cite[2.10.47-48]{federer})
we get that in all of the above-mentioned results instead of compactness
it is enough to assume that the set is analytic (or Borel) and on the
other hand, we can always suppose that the given compact or analytic set
is a compact $s$-set. Thus $C(n,\al)$ can be also expressed as
\begin{multline*}
C(n, \alpha)= \sup \{s : \exists A \subset \R^n
\textrm{ analytic such that } \\ \hdim(A)=s
\textrm{ and } A \textrm{ does not contain the angle } \alpha \},
\end{multline*}
or
\begin{multline*}
C(n, \alpha) = \sup \{s : \exists K \subset \R^n
\textrm{ compact such that } \\ 0<\Hd^s(K)<\infty
\textrm{ and } K \textrm{ does not contain the angle } \alpha \}.
\end{multline*}

However, as we prove it in the Appendix (Theorem~\ref{thm:transf}), some assumption about the set is
necessary, otherwise the above function would be $n$ for any $\al$.
In fact, for any given $n$ and $\al$ we construct by transfinite induction
a set in $\Rn$ with full Lebesgue outer measure that
does not contain the angle $\al$.

Note that in the definition 
of $\Ch(\al,\de)$ (Definition \ref{def:ch}), when we want to
find an angle in an open interval $(\al-\de,\al+\de)$, we have no 
assumption about the set $A$. This is simply because the closure of $A$
contains an angle in $(\al-\de,\al+\de)$ if and only if $A$ does, 
so in these problems we can always assume that $A$ is closed. 
Combining this with the above mentioned well-known fact that any analytic set
with positive $\Hd^s$ measure contains a compact $s$-set we get
\begin{multline} \label{eq:compactch}
\Ch(\alpha, \delta) = \sup \{s : \exists n\ \exists K \subset \R^n
\textrm{ compact such that } 0<\Hd^s(K)<\infty \\
\textrm{ and } K \textrm{ does not contain any angle from } 
(\alpha-\delta, \alpha+\delta) \} . 
\end{multline}

In fact, when we want to find an angle near to a given angle, then we get the
same results if we replace Hausdorff dimension by upper Minkowski dimension,
but this is not as clear as the above observation 
(see Corollary~\ref{cor:mink=haus}).

The following theorem,
which is the first statement of \cite[Theorem 10.11]{mattila},
plays essential role in some of our proofs.

\begin{notation}
The set of $k$-dimensional subspaces of $\Rn$ will be denoted by $G(n,k)$
and the natural probability measure on it by $\gamma_{n,k}$
(see e.g. \cite{mattila} for more details).
\end{notation}

\begin{thm}\label{thm:intersect2}
If $m<s<n$ and $A$ is an $\Hd^s$ measurable subset of $\R^n$ with $0<\Hd^s(A)<\infty$,
then
\[
\hdim\big(A\cap(W+x)\big)=s-m
\]
for $\Hd^s \times \gamma_{n,n-m}$ almost all $(x,W) \in A \times G(n,n-m)$.
\end{thm}

In two dimensions it says that for $\Hd^s$ almost all $x\in A$, almost all lines through $x$
intersect $A$ in a set of dimension $s-1$. One would expect that this theorem also holds for half-lines instead
of lines. Indeed, Marstrand proved it in \cite[Lemma 17]{marstrand}. Although the lemma only says that it holds for lines, he actually proves it for half-lines. Therefore the following theorem is also true.

\begin{thm}\label{thm:intersect_halflines}
Let $1<s<2$ and let $A\subset \R^2$ be $\Hd^s$ measurable with
$0<\Hd^s(A)<\infty$. For any $x\in \R^2$ and $\vartheta\in [0,360^\circ)$ let
$L_{x,\vartheta}$ denote the half-line from $x$ at angle $\vartheta$. Then
\[
\hdim\big(A\cap L_{x,\vartheta}\big)=s-1
\]
for $\Hd^s \times \lambda$ almost all
$(x,\vartheta) \in A \times [0,360^\circ)$.
\end{thm}

\section{Finding a given angle}
\label{given}

In this section we give estimates to $C(n,\al)$. For $n=2$ we
get the following exact result.

\begin{thm}\label{thm:bigsets_in_the_plane}
For any $\alpha \in [0,180^\circ]$ we have $C(2,\alpha)=1$.
\end{thm}
\begin{proof}
A line has dimension $1$ and it contains only the angles $0$ and $180^\circ$. A circle also has dimension $1$, but does not contain the angles $0$ and $180^\circ$. Therefore $C(2,\alpha)\ge 1$ for all $\alpha \in [0,180^\circ]$.

For the other direction let $\alpha \in [0,180^\circ]$ and $s>1$ fixed.
We have to
prove that any compact $s$-set contains the angle $\alpha$. By Theorem \ref{thm:intersect_halflines}, there exists $x\in K$ such that $\hdim(K\cap L_{x,\vartheta})=s-1$ for almost all $\vartheta\in [0,360^\circ)$, where
$L_{x,\vartheta}$ denotes the half-line from $x$ at angle $\vartheta$.
Hence we can take $\vartheta_1, \vartheta_2\in [0,360^\circ)$ such that $|\vartheta_1-\vartheta_2|=\alpha$, and $\hdim(K\cap L_{x,\vartheta_i})=s-1$ for $i=1,2$. If $x_i\in L_{x,\vartheta_i}\setminus \{x\}$ then the angle between the vectors $x_1-x$ and $x_2-x$ is $\alpha$, so indeed, $K$ contains the angle $\alpha$.
\end{proof}

An analogous theorem holds for higher dimensions.


\begin{thm}\label{n-1}
If $n\ge2$ and $\alpha \in [0,180^\circ]$ then $C(n,\alpha)\le n-1$.
\end{thm}
\begin{proof}
We have already seen the case $n=2$, so we may assume that $n\ge 3$.
It is enough to show that if $s>n-1$ and $K$ is a compact $s$-set, then $K$
contains the angle $\alpha$.
By Theorem \ref{thm:intersect2}, there exists $x \in K$ such that there exists a
$W \in G(n,2)$ with $\hdim(B)=s-n+2>1$ for $B\mathdef A\cap (W+a)$. The set $B$ lies in a two-dimensional plane, so we can think
about $B$ as a subset of $\R^2$. Applying Theorem
\ref{thm:bigsets_in_the_plane} completes the proof.
\end{proof}

Now we are able to give the exact value of $C(n,0)$ and $C(n,180^\circ)$.

\begin{thm}
$C(n,0)=C(n,180^\circ)=n-1$ for all $n\ge 2$.
\end{thm}
\begin{proof}
One of the inequalities was proven in the previous theorem, while the other one is shown by the $(n-1)$-dimensional sphere.
\end{proof}

We prove a better upper bound for $C(n,90^\circ)$.

\begin{thm}
\label{n/2}
If $n$ is even then $C(n,90^\circ)\le n/2$. If $n$ is odd then $C(n, 90^\circ)\le (n+1)/2$.
\end{thm}
\begin{proof}
First suppose that $n$ is even. Let $s>n/2$ and let $K$ be a compact $s$-set. From
Theorem \ref{thm:intersect2} we know that there exists a point $x\in K$ such that
\begin{equation}\label{eq:1}
\hdim\big(K \cap (x+W)\big)=s-n/2>0
\end{equation}
for $\gamma_{n,n/2}$ almost all $W\in G(n,n/2)$. There exists a $W\in G(n,n/2)$ such that (\ref{eq:1}) holds both
for $W$ and $W^\bot$. As $(x+W)\cap(x+W^\bot)=\{x\}$, by choosing a $y \in K \cap (x+W)$ and $z \in K \cap (x+W^\bot)$ such that $x\ne y$ and $x\ne z$, we find a right angle at $x$ in the triangle $xyz$.

Now suppose that $n$ is odd, $s>(n+1)/2$ and $K$ is a compact $s$-set. With a similar argument we can conclude that there exist
$x\in K$ and $W\in G(n,(n+1)/2)$ such that $\hdim\big(K \cap (x+W)\big)=s-(n+1)/2>0$ and $\hdim\big(K \cap (x+W^\bot)\big)=s-(n-1)/2>1$. If $y \in K \cap (x+W) \setminus \{x\}$ and $z \in K \cap (x+W^\bot)\setminus \{x\}$, then there is again a right angle at $x$ in the triangle $xyz$.
\end{proof}

\begin{rem}
By the following result of the fifth author (M\'ath\'e \cite{mathe})
the above estimate is sharp if $n$ is even:
for any $n$ there exists a compact set of Hausdorff dimension $n/2$
in $\Rn$ that does not contain $90^\circ$.
Therefore if $n$ is even, we have $C(n,90^\circ)=n/2$.

The construction uses number theoretic ideas and
even though the set contains angles arbitrarily close to $90^\circ$,
it succeeds to avoid the right angle.
In the next section we will present a different approach
where the constructed sets avoid not only a certain angle $\alpha$
but also a whole neighborhood of $\alpha$.
\end{rem}

\section{A self-similar construction}
\label{construction}

In this section we construct a self-similar set in $\R^n$
with large dimension such that it does not contain a certain angle
$\alpha \in (0,180^\circ)$.
On the negative side, our method does not work for the angles
$60^\circ$, $90^\circ$ and $120^\circ$. On the positive side,
the presented sets will avoid a whole neighborhood of $\alpha$, not only $\alpha$.

We start with two simple lemmas.
\begin{lemma}\label{thm:angles_in_simplices}
Let $P_0,\ldots, P_n$ be the vertices of a regular $n$-dimensional simplex. For any quadruples of indices $(i,j,k,l)$ with $i\ne j$ and $k\ne l$, the angle between the lines $P_iP_j$ and $P_kP_l$ is either $0$, $60^\circ$ or $90^\circ$.
\end{lemma}
\begin{proof}
The set $\{P_i,P_j,P_k,P_l\}$ is the set of vertices of a one-, two-, or three-dimensional regular simplex. Our assertion is clear in either of these cases.
\end{proof}

\begin{defin}
A self-similar set $K$ is said to satisfy the
\emph{strong separation condition} if there
exist similarities $S_0,\ldots,S_k$ such
that $K=S_0(K)\cup \cdots \cup S_k(K)$ and the sets
$S_i(K)$ are pairwise disjoint.

We say that the transformation $f:\R^n\to \R^n$ is a \emph{homothety} if $f$
is the identity or if $f$ has exactly one fixed point (say $O$), and there exists a nonzero real number $r$ such that for any point $P$ we have $f(P)-O=r(P-O)$. The point $O$ is called the \emph{center of the homothety}, and $r$ is called the \emph{ratio of magnification}. We call $K$ \emph{homothetic} if $S_i$ is a homothety for $i=0,1,\ldots,k$.
\end{defin}

\begin{lemma}\label{lemma:same-direction}
Let $K$ be a homothetic self-similar set,
in other words suppose that $K=S_0(K)\cup\ldots\cup S_m(K)$, where each
$S_i$ is a homothety.

Then, for any $x_0,x_1\in K$, $x_0\ne x_1$ there exist $y_0,y_1$ and
$i\ne j$ such that
$y_0\in S_i(K)$ and $y_1\in S_j(K)$ and $y_0-y_1$ is parallel to $x_0-x_1$.
\end{lemma}

\begin{proof}
Since $x_0,x_1\in K$, there exist sequences
$i_{0,1}, i_{0,2}, \ldots$ and $i_{1,1}, i_{1,2}, \ldots$ such that
\[
x_0\in S_{i_{0,1}}\Big(S_{i_{0,2}}\big(\cdots S_{i_{0,k}}(K)\big)\Big) \quad\textrm{and}\quad x_1\in S_{i_{1,1}}\Big(S_{i_{1,2}}\big(\cdots S_{i_{1,k}}(K_1)\big)\Big)
\]
for any positive integer $k$.

Let $k$ be the smallest positive integer such that $i_{0,k}\ne i_{1,k}$ (such a $k$ exists else $x_0$ and $x_1$
would coincide). Set
\[S\mathdef S_{i_{0,1}}\Big(S_{i_{0,2}}\big(\cdots S_{i_{0,k-1}}(\cdot)\big)\Big).
\]
There exist $y_0\in S_{i_{0,k}}(K)$ and $y_1\in S_{i_{1,k}}(K)$ such that $x_0=S(y_0)$ and $x_1=S(y_1)$. Since $S$ is also a homothety, $y_0-y_1$ is parallel to $x_0-x_1$.
\end{proof}

\begin{thm}\label{thm:selfsimilar}
For any $\varepsilon > 0$ there exists a constant $c_\varepsilon > 0$ such that 
for every $n\ge2$ there exists a compact homothetic self-similar set $K\subset \R^n$ 
with $\hdim(K) \ge c_\varepsilon \log n$ and with the property that 
all angles occurring in the set fall into the $\varepsilon$-neighborhood of
the angles $\{0, 60^\circ, 90^\circ, 120^\circ, 180^\circ \}$. 

In particular, for any 
$\alpha\in (0,180^\circ)\setminus \{60^\circ, 90^\circ, 120^\circ\}$ 
we construct a compact set of dimension $c(\alpha) \log n$ 
that does not contain the angle $\alpha$; 
moreover, the set even avoids a small neighborhood of $\alpha$.
\end{thm}

\begin{proof}
Our set $K$ will be a modified version of the Sierpi\'nski gasket.
Take a regular $n$-dimensional simplex with unit edge length in
$\R^n$, denote its vertices by $P_0, \ldots, P_n$ and let $K_1\mathdef \mathrm{conv}(\{P_0,\ldots, P_n\})$.
Fix a $0<\delta<1/2$ and denote by $S_i$ the homothety of ratio $\delta$ centered at $P_i$ ($i=0,\ldots,n$). The similarities $S_i$ ($i=0,\ldots,n$) uniquely determine a self-similar set $K$ which can also be written in the following form:
\[
K\mathdef \bigcap_{k=1}^\infty \bigcup_{(i_1,\ldots,i_k)\in \{0,\ldots,n\}^k} S_{i_1}\Big(S_{i_2}\big(\cdots
S_{i_k}(K_1)\big)\Big).
\]
The set $K$ clearly satisfies the strong separation condition. By \cite[Theorem 4.14]{mattila}, the dimension of $K$ is the unique positive number $s$ for which $(n+1)\delta^s=1$, therefore
\[
\hdim(K)=\frac{\log(n+1)}{\log \frac{1}{\delta}}.
\]

We say that a \emph{direction} $V\in G(n,1)$ \emph{occurs in a set} $H\subset \R^n$ if there are $x,y\in H$, $x\ne y$ such that $x-y$ is parallel to $V$. We will show that the directions occurring in $K$ are actually \emph{close} to the directions occurring in $\{P_0, \ldots, P_n\}$.

Let $V\in G(n,1)$ which occurs in $K$ and let $x_0,x_1\in K$, $x_0\ne x_1$ such that $x_0-x_1$ is parallel to $V$.
By Lemma~\ref{lemma:same-direction}
there exist $y_0,y_1\in K$, $y_0\ne y_1$ such that $y_0-y_1$ is also parallel to $V$ and there exist $i\ne j$ with $y_0\in S_i(K)$ and $y_1\in S_j(K)$.

We may assume without loss of generality that $y_0\in S_0(K)$, $y_1\in S_1(K)$.
We will show that the angle $\varphi$ between $y_0-y_1$ and $P_0-P_1$ is small, which is equivalent with $\cos \varphi$ being close to 1. Let $h_i = y_i - P_i$. We have $||h_i||\le\delta$ ($i=0,1$), hence
\[
\cos \varphi = \frac{\langle y_0-y_1, P_0-P_1\rangle}{||y_0-y_1||\cdot ||P_0-P_1||}=
\frac{1+\langle h_0-h_1, P_0-P_1\rangle}{||(P_0-P_1)+(h_0-h_1)||} \ge
\frac{1-2\delta}{1+2\delta}.
\]
Set $\varepsilon(\delta)=2\arccos(\frac{1-2\delta}{1+2\delta})$. Lemma \ref{thm:angles_in_simplices} implies that the angles occurring in $K$ are in the union of the following intervals: $[0,\varepsilon]$, $[60^\circ-\varepsilon,60^\circ+\varepsilon]$, $[90^\circ-\varepsilon,90^\circ+\varepsilon]$, $[120^\circ-\varepsilon,120^\circ+\varepsilon]$, $[180^\circ-\varepsilon,180^\circ]$. If $\delta$, and therefore $\varepsilon$ is sufficiently small, then neither of these intervals contain $\alpha$.
\end{proof}

The first author (Harangi \cite{harangi}) improved this result: 
he used the same methods to show that there exists a set with the same properties 
and with dimension $c_\varepsilon n$. 
Moreover, even for the angles $60^\circ$ and $120^\circ$
it is possible to construct large dimensional
homothetic self-similar sets avoiding these angles. 

However, as the next theorem shows, one cannot
avoid the right angle with similar constructions.

\begin{thm}
\label{rectangle}
Let $K\subset \R^n$ be a compact self-similar set.
Suppose that we have homotheties $S_0,\ldots,S_k$ with ratios less than 1
such that $K=S_0(K)\cup S_1(K) \cup \cdots \cup S_k(K)$ and
the sets $S_i(K)$ are pairwise disjoint
(that is, the strong separation condition is satisfied).
Then $K$ contains four points that form a non-degenerate rectangle
given that $\hdim(K)>1$.
\end{thm}
\begin{proof}
We begin the proof by defining the following map:
\[
D:\enskip K\times K \setminus \{(x,x):x\in K\}\to S^{n-1};\quad (x,y)\mapsto \frac{x-y}{||x-y||}.
\]
We denote the range of $D$ by $\mathrm{Range}(D)$.
The set $\mathrm{Range}(D)$ can be considered as
the set of directions in $K$.

First we prove that if $K$ is such a self-similar set then $\mathrm{Range}(D)$ is closed.
By Lemma~\ref{lemma:same-direction}, for any $x,y\in K$, $x\ne y$ there exist
$x'\in S_i(K)$ and $y'\in S_j(K)$ for some $i\neq j$ such that
$x'-y'$ is parallel to $x-y$.
If $d(\cdot,\cdot)$ denotes the Euclidean distance then
\[
\min_{0\le i<j\le k} d(S_i(K),S_j(K))=c>0,
\]
so $\mathrm{Range}(D)$  actually equals to the image of $D$ restricted to the set $K\times K\setminus \{(x,y): d(x,y)<c\}$. As this is a compact set, the continuous image is also compact, and so $\mathrm{Range}(D)$ is indeed compact.

Next we show that for any $v\in S^{n-1}$ there exist $x,y \in K$, $x\ne y$
such that the vectors $v$ and $D(x,y)$ are perpendicular. If this was not
true, the compactness of $\mathrm{Range}(D)$ would imply that the orthogonal
projection $p$ to a line parallel to $v$ would be a one-to-one map on $K$ with
$p^{-1}$ being a Lipschitz map on $p(K)$.
This would imply $\hdim(K)\le1$, which is a contradiction.

For simplifying our notation, let $f\mathdef S_0$, $g\mathdef S_1$.
The homotheties $f \circ g$ and $g \circ f$ have the same ratio. Denote their
fixed points by $P$ and $Q$, respectively.
Since $P\ne Q$, there are $x,y\in K$, $x\ne y$ such that $x-y$ is
perpendicular to $P-Q$. It is easy to check that the points $f(g(x))$,
$f(g(y))$, $g(f(y))$ and $g(f(x))$ form a non-degenerate rectangle.
\end{proof}


%
%
%
%
\section{Finding angles close to a given angle}
\label{approx}

Theorem~\ref{thm:selfsimilar} 
showed that for any angles $0<\al<180^\circ$ and $\de>0$ such that
$0, 60^\circ, 90^\circ, 120^\circ, 180^\circ \not\in (\al-\de, \al+\de)$ 
there exist compact
sets of arbitrarily large Hausdorff dimension that do not contain
any angle from $(\al-\de,\al+\de)$. That is, using the notation we
introduced in Definition~\ref{def:ch}, we have $\Ch(\al,\de)=\infty$ if
$\al \neq 0, 60^\circ, 90^\circ, 120^\circ, 180^\circ$ and $\de$ is small enough.
In this section we show that the other claims of Table~\ref{table2} of the
introduction also hold. 

We start by proving that a set that does not contain angles near to
$90^\circ$ must be very small, it cannot have Hausdorff dimension bigger
than $1$. This makes $90^\circ$ very special since
the analogous statement would be false for any other angle
$\al\in(0,180^{\circ})$ (see Theorem~\ref{thm:selfsimilar} and
Remark~\ref{rem:sharp}). 
This result is clearly sharp since a line segment contains only $0$ and
$180^{\circ}$.

\begin{thm}
\label{Hausdorff1}
Any set $A\su\R^n$ $(n\ge 2)$
with Hausdorff dimension greater than 1 contains
angles arbitrarily close to the right angle.
Thus $\Ch(90^{\circ},\de)=1$ for  any $\de>0$.

\end{thm}
\begin{proof}
By the equivalent definition (\ref{eq:compactch}) of $\Ch$ we found
in the introduction
we can assume that $A$ is compact and 
$0<\Hd^s(A)<\infty$ for some $s>1$.
Applying Theorem \ref{thm:intersect2} for $m=1$
we obtain that for $\Hd^s$ almost all $x \in A$
the set $A\cap(W+x)$ has positive dimension
for $\gamma_{n,n-1}$ almost all $W \in G(n,n-1)$.
Let us fix a point $x$ with this property
and let $y \ne x$ be an arbitrary point in $A$.


Since for any fixed $\delta>0$
the subspaces forming an angle at least $90^\circ - \delta$ with $x-y$
have positive measure, and the exceptional set in
Theorem~\ref{thm:intersect2} is of measure zero, the theorem follows.
\end{proof}

Now we prove the same result for upper
Minkowski dimension instead of Hausdorff dimension.
It is well-known that the upper Minkowski dimension
is always greater or equal than the Hausdorff dimension.
Hence the following theorem is stronger than the previous one.
\begin{thm} \label{thm:90_mink}
Any bounded set $A$ in $\R^n$ $(n\ge 2)$
with upper Minkowski dimension greater than 1 contains
angles arbitrarily close to the right angle.
\end{thm}
The upper Minkowski dimension can be defined in many different ways,
we will use the following definition (see \cite[Section 5.3]{mattila} for details).
\begin{defin}\label{def:mink}
By $B(x,r)$ we denote the closed ball with center $x\in \R^n$ and radius $r$.
For a non-empty bounded set $A\subset\R^n$
let $P(A, \varepsilon)$ denote the greatest integer $k$ for which
there exist disjoint balls $B(x_i,\varepsilon)$ with $x_i\in A$, $i=1,\dots,k$.
The \emph{upper Minkowski dimension} of $A$ is defined as
$$ \mdim(A)\mathdef\sup\{s:
\limsup_{\varepsilon\rightarrow 0+}P(A,\varepsilon)\varepsilon^s=\infty\} .$$
Note that we get an equivalent definition if
we consider the $\limsup$ for $\varepsilon$'s
only of the form $\varepsilon = 2^{-k}$, $k\in \N$.
\end{defin}

The following technical lemma is needed not only for the proof of 
Theorem~\ref{thm:90_mink} but also for the result
about finding angles near to $60^{\circ}$.
It roughly says that in a set of large upper Minkowski dimension
one can find many points such that
the distance of each pair is more or less the same.

\begin{lemma} \label{lem:mink}
Suppose that $\mdim(A) > t$
for a bounded set $A\subset\R^n$ and a positive real $t$.
Then for infinitely many positive integers $k$
it holds that for any integer $0 < l < k$
there are more than $2^{(k-l)t}$ points in $A$
with the property that the distance of any two of them
is between $2^{-k+1}$ and $2^{-l+2}$.
\end{lemma}
\begin{proof}
Let
$$ r_k=P(A, 2^{-k}) 2^{-kt}. $$
Due to the previous definition $\limsup_{k\rightarrow\infty} r_k = \infty$.
It follows that there are infinitely many values of $k$
such that $r_k > r_l$  for all $l < k$.
Let us fix such a $k$ and let $0<l<k$ be arbitrary.

By the definition of $r_k$, there are
$r_k2^{kt}$ disjoint balls with radii $2^{-k}$ and centers in $A$.
Let $\mathcal{S}$ denote the set of the centers of these balls.
Clearly the distance of any two of them is at least $2^{-k+1}$.

Similarly, we can find a maximal system of disjoint balls
$B(x_i,2^{-l})$ with $x_i\in A$, $i=1,\dots,r_l2^{lt}$.
Consider the balls $B(x_i,2^{-l+1})$ of doubled radii.
These doubled balls are covering the whole $A$
(otherwise the original system would not be maximal).
By the pigeonhole principle,
one of these doubled balls contains at least
$$\frac{r_k2^{kt}}{r_l2^{lt}}=\frac{r_k}{r_l}2^{(k-l)t} > 2^{(k-l)t}$$
points of $\mathcal{S}$. These points clearly have the desired property.
\end{proof}
Now we are in a position to prove the theorem.
\begin{proof}[Proof of Theorem~\ref{thm:90_mink}]
We can assume that $\diam(A)>2$.
Fix a $t$ such that $\mdim(A)>t>1$.

Lemma \ref{lem:mink} tells us that there are arbitrarily large integers $k$
such that for any $l<k$ one can have more than $2^{(k-l)t}$ points in $A$
such that each distance is between $2^{-k+1}$ and $2^{-l+2}$.
Let $\mathcal{S}$ be a set of such points
and pick an arbitrary point $O\in \mathcal{S}$.
Since $\diam(A)>2$, there exists a point $P\in A$ with $OP\ge1$.
Now we project the points of $\mathcal{S}$ to the line $OP$.
There must be two distinct points $Q_1, Q_2 \in \mathcal{S}$
such that the distance of their projection is at most
$$\frac{2^{-l+2}}{2^{(k-l)t}}= 2^{-l+2-(k-l)t} ,$$
It follows that
$$\cos \angle(\overrightarrow{Q_1Q_2},\overrightarrow{OP})\le
\frac{2^{-l+2-(k-l)t}}{2^{-k+1}}=2^{-(k-l)(t-1)+1}.$$
Since $Q_1O \le 2^{-l+2}$ and $OP\ge1$,
the  angle of the lines  $OP$ and $Q_1P$
is at most  $C_1 2^{-l}$ with some constant $C_1$.
Combining the previous results we get that
$$ \left| \angle P Q_1 Q_2 - 90^\circ \right| \le
C_1 2^{-l} + C_2 2^{-(k-l)(t-1)} $$
with some constants $C_1, C_2$.
The right hand side can be arbitrarily small
since $t-1$ is positive and both $l$ and $k-l$ can be chosen to be large.
\end{proof}

Now we try to find angles close to $60^\circ$.
We will do that by finding three points forming an \emph{almost regular} triangle
provided that the dimension of the set is sufficiently large.

We will need a simple result from Ramsey theory.
Let $R_r(3)$ denote the least positive integer $k$
for which it holds that no matter how we color
the edges of a complete graph on $k$ vertices with $r$ colors
it contains a monochromatic triangle.
The next inequality can be obtained easily:
$$ R_r(3) \leq r \cdot R_{r-1}(3) - (r-2) .$$
(A more general form of the above inequality
can be found in e.g.~\cite[p.~90, Eq.~2]{ramsey}.)
It readily implies the following upper bound for $R_r(3)$.
\begin{lemma} \label{lem:ramsey}
For any positive integer $r \geq 2$
$$ R_r(3) \leq 3r! ,$$
that is, any complete graph on at least $3r!$ vertices
edge-colored by $r$ colors contains a monochromatic triangle.
\end{lemma}
Using this lemma we can prove the following theorem.
\begin{thm}\label{thm:tri}
There exists an absolute constant $C$ such that
whenever $\mdim(A)> \frac{C}{\delta}\log(\frac{1}{\delta})$
for some bounded set $A\subset\R^n$ and $\delta >0$ the following holds:
$A$ contains three points that form a $\delta$-almost regular triangle,
that is, the ratio of the length of the longest and shortest sides
is at most $1 + \delta$.
\end{thm}
As an immediate consequence, we can find angles close to $60^\circ$.
\begin{cor} \label{cor:60}
Suppose that $\mdim(A)> \frac{C}{\delta}\log(\frac{1}{\delta})$
for some bounded set $A\subset\R^n$ and $\delta >0$.
Then $A$ contains angles from the interval $(60^\circ - \delta, 60^\circ]$
and also from $[60^\circ, 60^\circ+\delta)$.
Therefore $\Ch(60^\circ,\de)\le \frac{C}{\delta}\log(\frac{1}{\delta})$.
\end{cor}
\begin{rem}\label{rem:sharp}
The above theorem and even the corollary is essentially sharp:
the first author (Harangi \cite{harangi}) constructed a set with Hausdorff dimension
$\frac{c}{\delta}/\log(\frac{1}{\delta})$
and without any angles from the interval
$(60^\circ-\delta, 60^\circ+\delta)$, so
we have $\Ch(60^\circ,\de)\ge \frac{c}{\delta}/\log(\frac{1}{\delta})$.
\end{rem}
\begin{proof}[Proof of Theorem \ref{thm:tri}]
Let $t=\frac{C}{\delta}\log(\frac{1}{\delta})$ and
apply Lemma \ref{lem:mink} for $l=k-1$.
We obtain at least $2^{t}$ points in $A$ such that
each distance is in the interval $[2^{-k+1},2^{-k+3}]$.
Let $a=2^{-k+1}$ and divide $[a,4a]$ into
$N=\lceil\frac{3}{\delta}\rceil$ disjoint intervals of length at most $\delta a$.
Regard the points of $A$ as the vertices of a graph.
Color the edges of this graph with $N$ colors according to
which interval contains the distance of the corresponding points.

Easy computation shows that $2^t>3N!$ (with a suitable choice of $C$).
Therefore the above graph contains
a monochromatic triangle by Lemma \ref{lem:ramsey}.
It easily follows that the three corresponding points
form a $\delta$-almost regular triangle in $\R^n$.
\end {proof}
\begin{rem}
The same proof yields the following:
for any positive integer $d$ and positive real $\delta$
there is a number $K(d, \delta)$ such that
whenever $\mdim(A)>K(d, \delta)$ for some bounded set $A$,
one can find $d$ points in $A$ with the property that
the ratio of the largest and the smallest distance
among these points is at most $1+\delta$.
(One needs to use the fact that the Ramsey number $R_r(d)$ is finite.)
\end{rem}
In order to derive similar results for $120^\circ$ instead of $60^\circ$
we show that if large Hausdorff dimension implies the existence of an angle near $\alpha$,
then it also implies the existence of an angle near $180^\circ-\alpha$.
\begin{prop} \label{prop:suppl_angles}
Suppose that $s=s(\alpha, \delta, n)$ is a positive real number such that
any analytic set $A\subset\R^n$ with $\Hd^s(A)>0$ contains an angle
from the interval $(\alpha-\delta, \alpha+\delta)$.
Then any analytic set $B\subset\R^n$ with $ \Hd^s(B)>0$
contains an angle from the interval $(180^\circ-\alpha-\delta',180^\circ-\alpha+\delta')$
for any $\de' > \de$.
\end{prop}
\begin{proof}
Again, we can assume that $0<\Hd^s(B)<\infty$.
It is well-known that for $\Hd^s$ almost all $x\in B$
the set $B \cap B(x,r)$ has positive $\Hd^s$ measure for any $r>0$
\cite[Theorem 6.2]{mattila}.
If we omit the exceptional points from $B$,
this will be true for every point of the obtained set.
Assume that $B$ had this property in the first place.
Then, by the assumptions of the proposition,
any ball around any point of $B$
contains an angle from the $\delta$-neighborhood of $\alpha$.

We define the points $P_m, Q_m, R_m \in B$ recursively
in the following way. Fix a small $\varepsilon$.
First take  $P_0, Q_0, R_0$ such that the angle $\angle P_0Q_0R_0$
falls into the interval $(\alpha-\delta, \alpha+\delta)$.
If the points $P_m, Q_m, R_m$ are given, then choose points
$P_{m+1},Q_{m+1}, R_{m+1}$ from the
$\varepsilon\cdot\min(Q_mP_m,Q_mR_m)$-neighborhood of $P_m$ such that
$\angle P_{m+1}Q_{m+1}R_{m+1} \in (\alpha-\delta, \alpha+\delta)$.

We can find two indices $k > l$ such that the
angle enclosed by the vectors $\overrightarrow{Q_lP_l}$ and
$\overrightarrow{Q_kP_k}$ is less than $\varepsilon$.
It is clear that if we choose $\varepsilon$ sufficiently small, then
$\angle(Q_l, Q_k,R_k)\in(180^\circ-\alpha-\delta', 180^\circ-\alpha+\delta')$.
\end{proof}
\begin{rem}\label{rem:closed}
Proposition~\ref{prop:suppl_angles} holds for $\delta'=\delta$ as well.
Surprisingly, it even holds for some $\delta'<\delta$.
The reason behind is the following.
If every analytic set $A\subset \R^n$ with $\Hd^s(A)>0$ contains an angle
from the interval $(\alpha-\delta, \alpha+\delta)$, then there necessarily exists a closed subinterval
$[\alpha-\gamma, \alpha+\gamma]$ ($\gamma<\delta$) such that
every analytic set $A\subset \R^n$ with $\Hd^s(A)>0$ contains an angle
from the interval $[\alpha-\gamma, \alpha+\gamma]$. We prove this statement in the Appendix (Theorem~\ref{thm:closed}).

This implies that $\Ch$ 
satisfies the symmetry property
$$\Ch(\alpha, \delta)=\Ch(180^\circ-\alpha, \delta).$$
\end{rem}

\begin{thm}\label{thm:120}
There exists an absolute constant $C$ such that
any analytic set $A\subset\R^n$ with
$\hdim(A) > \frac{C}{\delta}\log(\frac{1}{\delta})$
contains an angle from the $\delta$-neighborhood of $120^\circ$.
Therefore $\Ch(120^\circ,\de)\le \frac{C}{\delta}\log(\frac{1}{\delta})$.
\end{thm}
\begin{proof}
The claim readily follows from
Corollary \ref{cor:60}, Proposition \ref{prop:suppl_angles}
and the fact that the upper Minkowski dimension is
greater or equal than the Hausdorff dimension.
\end{proof}

\begin{rem}
In fact, as for the other angles, 
in Theorem~\ref{thm:120} it is also enough to assume that
the upper Minkowski dimension is bigger than 
$\frac{C}{\delta}\log(\frac{1}{\delta})$. This follows from a more general
result that we prove in the Appendix.
\end{rem}

To find angles arbitrarily close to $0$ and $180^\circ$,
it suffices to have infinitely many points.
\begin{prop} \label{prop:pi_angle}
Any $A\subset \R^n$ of infinite cardinality
contains angles arbitrarily close to $0$ and
angles arbitrarily close to $180^\circ$.
Therefore $\Ch(0,\de)=\Ch(180^\circ,\de) = 0$.
\end{prop}

\begin{proof}[Sketch of the proof]
We claim that given $N$ points in $\R^n$
they must contain an angle less than
$\delta_1= \frac{C}{\root{n-1}\of{N}}$
and an angle greater than $180^\circ-\delta_2$
with $\delta_2=\frac{C}{\root{n-1}\of{\log N}}$.
The former follows easily from the pigeonhole principle.
The latter is a result of Erd\H os and
F\"uredi \cite[Theorem~4.3]{erdos/furedi}.
\end{proof}

\section{Appendix}

\subsection{A transfinite construction}

We prove the following theorem, which shows 
that if we allowed arbitrary sets in Definition~\ref{def:basicdef} 
then $C(n,\alpha)$ would be $n$.


\begin{thm}\label{thm:transf}
Let $n\ge 2$. 
For any $\alpha \in [0,180^\circ]$ there exists $H\subset \R^n$ such
that $H$ does not contain the angle $\alpha$, and $H$ has full Lebesgue
outer measure; that is, its complement does not contain any measurable set
with positive measure.
In particular, $\hdim(H)=n$.
\end{thm}

The proof we present here is shorter than our original proof, this one 
was suggested by Marianna Cs\"ornyei.

We need the following simple lemma, which might be well-known even for more
general sets but for completeness we present a proof. 
Recall that an algebraic set is the 
set of solutions of a system of polynomial equations.

\begin{lemma}\label{l:surfaces}
Less than continuum many proper algebraic subsets of $\R^n$ cannot cover
a Borel set of positive $n$-dimensional Lebesgue measure.
\end{lemma}

\begin{proof}
We prove by induction. For $n=1$ this is clear since proper algebraic subsets
of $\R$ are finite and every Borel set of positive Lebesgue measure 
has cardinality
continuum. 

Suppose that the lemma holds for $n-1$ but it is false for $n$, 
so there exists a collection $\mathcal A$ of 
less than continuum many proper algebraic subsets of $\R^n$ such that 
they cover a Borel set $B\su\R^n$ with positive Lebesgue measure.

Let $H^t$ denote the ``horizontal'' section 
$H^t=\{(x_1,\ldots,x_{n-1}):(x_1,\ldots,x_{n-1},t) \in H\}$ of a set
$H\su\R^n$ at ``height'' $t\in\R$. 
If $A$ is a proper algebraic subset of $\R^n$ then with a finite exception
every $A_t$ is a proper algebraic subset of $\R^{n-1}$.
Therefore, by using the assumption that the lemma holds for $n-1$,
we get that $(\cup\mathcal{A})_t$ can contain Borel sets of positive 
$n-1$-dimensional Lebesgue measure only for less than continuum many $t$.
Let $f(t)$ denote the $(n-1)$-dimensional Lebesgue measure of the Borel set
$B_t$. Since $B\su\cup\mathcal{A}$, we obtain that 
$\{t: f(t)>0\}$ has cardinality less than continuum. 

On the other hand, by Fubini theorem $f$ is measurable and its integral is
the Lebesgue measure of $B$, so it is positive. 
This implies that $\{t: f(t)>0\}$ is a measurable set of positive measure,
so it must have cardinality continuum, contradiction.
\end{proof}

\begin{proof}[Proof of Theorem~\ref{thm:transf}]
Take a well-ordering $\{B_\beta : \beta < \cont\}$ of the Borel subsets of
$\R^n$ with positive $n$-dimensional Lebesgue measure. 
We will construct a sequence of points $\{x_\beta : \beta < \cont \}$ of $\Rn$
using transfinite induction
so that 
\begin{equation}\label{condition}
x_{\be}\in B_\be \qquad \textrm{and} \qquad H_\beta= \{x_\de: \de\le\beta\}
\textrm{ does not contain the angle } \al
\end{equation}
 for any $\beta<\cont$.
This will complete the proof since then $H=\{x_\beta:\beta < \cont\}$
will have all the required properties.

Suppose that $\ga < \cont$ and we have already properly defined $x_\be$
for all $\be<\ga$ so that (\ref{condition}) holds for all
$\be<\ga$. 
For any $p,q\in \R^n$, $p\neq q$, let $A_{p,q}$
denote the set of those $x\in\R^n$ for which one of the angles of the triangle
$pqx$ is $\alpha$. Note that $A_{p,q}$ can be covered by three proper
algebraic subsets of $\R^n$. Then, by Lemma~\ref{l:surfaces}, 
the sets $A_{x_\de,x_{\de'}}$ $(\de,\de'<\ga, x_\de\neq x_{\de'})$ cannot
cover $B_\ga$, so we can choose a point 
$$
x_\ga\in B_\ga\setminus \cup\{A_{x_\de,x_{\de'}}\ :\ 
                       \de,\de'<\ga,\ x_\de\neq x_{\de'}\}.
$$
Then (\ref{condition}) also holds for $\beta=\ga$.

This way we obtain a sequence $(x_\be)_{\be<\cont}$, so that 
(\ref{condition}) holds for all $\beta<\cont$, which completes the proof.
\end{proof}

\subsection{The size of the neighborhood in the approximative problems}

Now, our goal is to prove the following theorem, which was claimed in Remark~\ref{rem:closed}.

\begin{thm}\label{thm:closed}
Suppose that $s=s(\alpha, \delta, n)$ is a positive real number such that every analytic set $A\subset\R^n$ with $\Hd^s(A)>0$ contains an angle
from the interval $(\alpha-\delta, \alpha+\delta)$.
Then there exists a closed subinterval
$[\alpha-\gamma, \alpha+\gamma]$ ($\gamma<\delta$) such that
every analytic set $A\subset\R^n$ with $\Hd^s(A)>0$ contains an angle
from the interval $[\alpha-\gamma, \alpha+\gamma]$.
\end{thm}

To prove this theorem, we need two lemmas. For $r \in (0,\infty]$ let
$$\Hd^s_r(A)=\inf \left\{ \sum_{i=1}^\infty \diam(U_i)^s \,:\, \diam(U_i)\le r, \ A\subset \cup_{i=1}^\infty U_i\right\},$$
thus $\Hd^s(A)=\lim_{r\to 0+} \Hd^s_r(A)$.

\begin{lemma}\label{lemma:limit}
Let $A_i$ be a sequence of compact sets converging in the Hausdorff metric to a set $A$. Then the following two statements hold.
\begin{itemize}
\item[(i)] $\Hc^s(A) \ge \limsup_{i\to\infty} \Hc^s(A_i).$
\item[(ii)] Suppose that for every $i=1,2,\ldots$ the set $A_i$ does not contain any angle from $[\alpha-\delta+1/i, \ \alpha+\delta-1/i]$. Then $A$ does not contain any angle from $(\alpha-\delta, \,\alpha+\delta)$.
\end{itemize}
\end{lemma}
\begin{proof}The first statement is well-known and easy.
To prove the second, notice that for any three points $x,y,z$ of $A$ there exist three points in $A_i$ arbitrarily close to $x,y,z$, for sufficiently large $i$.
\end{proof}
The next lemma follows easily from \cite[Theorem 2.10.17~(3)]{federer}.
For the sake of completeness, we give a short direct proof.
\begin{lemma}\label{lemma:surusegi}
Let $A\subset\R^n$ be a compact set satisfying $\Hd^s(A)>0$. Then there exists a ball $B$ such that $\Hc^s(A\cap B) \ge c\,\diam(B)^s$,
where $c>0$ depends only on $s$.
\end{lemma}
\begin{proof}
We may suppose without loss of generality that $\Hd^s(A)<\infty$. (Otherwise we choose a compact subset of $A$ with positive and finite $\Hd^s$ measure. If the theorem holds for a subset of $A$ then it clearly holds for $A$ as well.)

Choose $r>0$ so that $\Hd^s_r(A)> \Hd^s(A)/2$.
Cover $A$ by sets $U_i$ of diameter at most $r/2$ such that $\sum_i \diam(U_i)^s \le 2 \Hd^s(A)$.
Cover each $U_i$ by a ball $B_i$ of radius at most the diameter of $U_i$.
Then the balls $B_i$ cover $A$, have diameter at most $r$, and $\sum_i \diam(B_i)^s \le 2^{1+s} \Hd^s(A)$.

We claim that one of these balls $B_i$ satisfies the conditions of the Lemma for $c=2^{-2-s}$. Otherwise we have
$$\Hc^s(A\cap B_i) < 2^{-2-s} \, \diam(B_i)^s$$
for every $i$. Since the sets $A\cap B_i$ have diameter at most $r$, clearly $\Hd^s_r(A\cap B_i) = \Hc^s(A\cap B_i)$.
Therefore
\begin{multline*}
\Hd^s_r(A)\le \sum_i \Hd^s_r(A\cap B_i) <\sum_i 2^{-2-s} \, \diam(B_i)^s \\ \le 2^{-2-s} 2^{1+s} \Hd^s(A) = \Hd^s(A)/2,
\end{multline*}
which contradicts the choice of $r$.
\end{proof}

\begin{proof}[Proof of Theorem~\ref{thm:closed}]
Suppose on the contrary that there exist compact sets $K_i\subset \R^n$ with $\Hd^s(K_i)>0$ such that $K_i$ does not contain any angle from $[\alpha-\delta+1/i, \ \alpha+\delta-1/i]$.
Choose a ball $B_i$ for each compact set $K_i$ according to Lemma~\ref{lemma:surusegi}. Let $B$ be a ball of diameter $1$. Let $K_i'$ be the image of $K_i\cap B_i$ under a similarity transformation which maps $B_i$ to the  ball $B$. Thus $\Hc^s(K_i')\ge c$. Let $K$ denote the limit of a convergent subsequence of the sets $K_i$. We can apply Lemma~\ref{lemma:limit} to this subsequence and obtain $\Hc^s(K)\ge c$, implying $\Hd^s(K)>0$. Also, $K$ does not contain any angle from the interval $(\alpha-\delta, \,\alpha+\delta)$, which is a contradiction.
\end{proof}


\subsection{Replacing Hausdorff dimension by upper Minkowski dimension} 
Our final goal is showing that in the problems when we want angles only in
a neighborhood of a given angle, Hausdorff dimension can be replaced by 
Minkowski dimension. This will follow from the following theorem.
As pointed out by Pablo Shmerkin, this theorem also follows 
from a result of Furstenberg \cite{furstenberg}. 
His result is much more general and it is not immediately trivial to see 
that it implies what we need. 
Therefore we give a direct self-contained proof.

\begin{thm}\label{t1}
Let $A\subset \R^d$ be a bounded set with upper Minkowski dimension $s>0$. 
Then there exists a compact set $K$ of Hausdorff dimension $s$ such that 
all finite subsets of $K$ are limits of 
homothetic copies of finite subsets of $A$. 
(That is, for every finite set $S\subset K$ and $\eps>0$ 
there exists a set $S'\subset A$ and $r>0$, $t\in \R^d$ such that 
the Hausdorff distance of $t+rS'$ and $S$ is at most $\eps$.)
\end{thm}


Applying Theorem~\ref{t1} to a bounded set $A$ that does not contain any angle
from an open interval we get a compact set $K$
with the same property and with $\hdim(K)=\mdim(A)$.
Thus we get the following.

\begin{cor}\label{cor:mink=haus}
For any $\alpha \in [0,180^\circ]$ and $\de>0$, we have
\begin{multline*} 
\Ch(\alpha, \delta) = \sup \{\mdim(A) :
A \subset \R^n \textrm{ for some } n; A \textrm{ is bounded};\\
A \mbox{ does not contain any angle from } (\alpha-\delta, \alpha+\delta) \} .
\end{multline*}
\end{cor}


\begin{proof}[Proof of Theorem~\ref{t1}]
We will need to use a slightly different version of 
the Hausdorff content $\Hd^s_\infty(B)$ in this proof. 
Instead of covering $B \subset \R^d$ with arbitrary sets, 
we will only consider coverings with homothetic copies of 
the unit cube $[0,1]^d$. (From now on, 
a cube is always assumed to be a homothetic copy of the unit cube.) 
For a cube $C$, $\diam(C)$ is just the constant multiple of 
the edge length of $C$ (denoted here by $|C|$). 
For the sake of simplicity, we will use $|C|$ in our definition:  
for any $B \subset \R^d$ and $s>0$ let 
$$ \Hh^s_\infty(B) \mathdef \inf \left\{ \sum_{i=1}^\infty |C_i|^s : C_i 
\mbox{ is a cube for each } i; \ B \subset \bigcup_{i=1}^\infty C_i \right\} .$$
It is easy to see that 
$d^{-s/2} \Hd^s_\infty \leq \Hh^s_\infty \leq \Hd^s_\infty$. 
Also note that $ \Hh^s_\infty([0,1]^d) = 1$ for any $0 < s \leq d$. 

We may assume that $A\subset [0,1]^d$. 
For a positive integer $n$ we divide the unit cube into 
$n^d$ subcubes of edge length $1/n$. 
Let $A_n$ be the union of the subcubes that intersect $A$. 


We claim that for any fixed $0< \delta < s/2$, 
for infinitely many $n$ (depending on $\delta$) 
there exists a cube $C$ such that 
\begin{equation}\label{theclaim}
 |C| \geq n^{\frac{\delta}{2d}} / n  \mbox{ and } 
\Hh^{s-2\delta}_\infty(C \cap A_n) \geq 2^{-s-2} |C|^{s-2\delta} .
\end{equation}

First we show how the theorem follows from this claim. If \eqref{theclaim} holds for $n$ and $C$, then let $K_n$ be the image of $C \cap A_n$ under the homothety 
that maps $C$ to $[0,1]^d$. 
Hence $\Hh^{s-2\delta}_\infty(K_n)\ge 2^{-s-2}$. 
If $S\subset K_n$ is finite, then there exists $S'$ 
such that the Hausdorff distance of $S$ and $S'$ is 
at most $\sqrt{d} n^{-\delta/(2d)}$ and 
a homothetic image of $S'$ is in $A$. 

For each $\delta=1/l$ choose $n=n_l\ge l^l$ such that the claim holds.
Let $\tilde K$ be the limit of a convergent subsequence of $K_{n_l}$. 
By Lemma \ref{lemma:limit} the Hausdorff dimension of $\tilde K$ is at least 
$s$.
Let $K$ be a compact subset of $\tilde K$ of Hausdorff dimension $s$.
It is easy to check 
that $K$ satisfies all the required properties.

It remains to prove the claim. 
Since $\mdim(A) = s$, $A_n$ contains at least 
$n^{s-\delta}$ subcubes for infinitely many $n$. 
Fix such an $n$ with $n\ge 2^{4/\delta}$. Let 
$$ c=\min \left\{\Hh^{s-2\delta}_\infty(B) / m: 
B \mbox{ is the union of $m$ subcubes of } A_n, \ m\ge 1\right\} .$$
Since the unit cube covers $A_n$, by choosing $B$ as the union of 
$m\ge n^{s-\de}$ subcubes of $A_n$ we get 
$c\le  \Hh^{s-2\delta}_\infty(B)/m \le 1/n^{s-\delta}$.
(On the other hand, one subcube has content $1/n^{s-2\delta}$, hence 
the minimum is taken for a set $B$ for which $m$ is at least $n^\delta$.)

Suppose now that $B$ is a set for which the minimum is taken; that is,
$$\Hh^{s-2\delta}_\infty(B)=cm,$$
where $B$ consists of $m$ subcubes of $A_n$.
It follows that there exists a covering of $B$ with cubes $C_i$ ($i=1,2,\ldots$) such that 
$$\sum_{i=1}^\infty |C_i|^{s-2\delta} \le 2cm.$$

Let $k=n^{\delta/(2d)}$.
We say that a cube $C_i$ is ``bad'' if $|C_i|<k/n$, and ``good'' otherwise. 
The total volume of the bad cubes is at most
\begin{align*}
& \sum_{C_i \text{ is bad}} |C_i|^d = 
\sum_{C_i \text{ is bad}} |C_i|^{d-s+2\delta} |C_i|^{s-2\delta} \le 
(k/n)^{d-s+2\delta} \sum_{i=1}^\infty |C_i|^{s-2\delta} \\
& \le 2cm (k/n)^{d-s+2\delta} \le 2m k^{d-s+2\delta} n^{-\delta-d}
\le 2m k^{d} n^{-\delta-d} = 2m n^{-\frac{\delta}{2}-d}\le \frac{m}2 n^{-d},
\end{align*}
where the last four estimates follow from $c \le 1/n^{s-\delta}$,
$\de<s/2$, $k=n^{\delta/(2d)}$ and  $n\ge 2^{4/\delta}$.
So there are at most $m/2$ subcubes that are fully covered by bad cubes. 
Let $B'$ be the union of the remaining (at least $m/2$) subcubes in $B$. 
Since each subcube in $B'$ must intersect a good cube $C_i$, 
it follows that the cubes $2C_i$ cover $B'$, where 
$2C_i$ is the cube with the same center as $C_i$ and double edge length. 

Then the definition of $c$ implies that
$$ 
\sum_{C_i \text{ is good}} \Hh^{s-2\delta}_\infty(2C_i \cap A_n) \ge 
\Hh^{s-2\delta}_\infty(B')\ge c\frac{m}2 .$$ 
On the other hand, we have
$$ \sum_{C_i \text{ is good}} |2C_i|^{s-2\delta} \le 
2^{s-2\delta} \sum_{i=1}^\infty |C_i|^{s-2\delta} \le 2^{s-2\delta} 2cm 
\le 2^{s+1}cm.$$
Therefore there exists a good cube $C_i$ such that  
$$ \Hh^{s-2\delta}_\infty(2C_i \cap A_n) \ge 
2^{-s-2} |2C_i|^{s-2\delta} .$$ 
Thus (\ref{theclaim}) holds for the cube $C = 2 C_i$,
which completes the proof. 
\end{proof}

\end{document}